\documentclass[12pt]{article}
\usepackage[a4paper,margin=2.5cm]{geometry}
\usepackage[comma, round]{natbib}
\usepackage{float}
\usepackage{amssymb}
\usepackage{amsthm}
\usepackage{amsmath}
\usepackage{stmaryrd}
\usepackage{titletoc}
\usepackage{mathrsfs}
\usepackage[colorlinks, allcolors=blue]{hyperref}
\usepackage{amsfonts}

\usepackage{graphicx}
\usepackage[linesnumbered, ruled, vlined]{algorithm2e}
\usepackage{esint}
\usepackage{bm}

\numberwithin{equation}{section}
\numberwithin{figure}{section}

\theoremstyle{plain}
\newtheorem{thm}{Theorem}[section]

\newtheorem{lem}[thm]{Lemma}
\newtheorem{prop}[thm]{Proposition}

\theoremstyle{definition}

\newtheorem{asmp}{Assumption}[section]
\newtheorem{rmk}{Remark}[section]
\newtheorem{exam}{Example}[section]

\newcommand{\bR}{\mathbf{R}}

\newcommand{\bP}{\mathbb{P}}

\newcommand{\cB}{\mathcal{B}}
\newcommand{\cC}{\mathcal{C}}

\newcommand{\cE}{\mathcal{E}}
\newcommand{\cF}{\mathcal{F}}
\newcommand{\cL}{\mathcal{L}}
\newcommand{\cP}{\mathscr{P}}
\newcommand{\cS}{\mathcal{S}}
\newcommand{\cM}{\mathcal{M}}
\newcommand{\cN}{\mathcal{N}}

\newcommand{\cW}{\mathcal{W}}

\newcommand{\cX}{\mathcal{X}}

\newcommand{\eps}{\varepsilon}

\newcommand{\la}{\langle}
\newcommand{\ra}{\rangle}

\newcommand{\md}{\mathop{}\mathopen\mathrm{d}}
\newcommand{\me}{\mathop{}\mathopen\mathrm{e}}
\newcommand{\E}{\mathbb{E}}

\newcommand{\im}{\mu_{*}}
\newcommand{\wl}{\pi}
\renewcommand{\wp}{\varpi}

\title{Empirical approximation to invariant measures for McKean--Vlasov processes: mean-field\\interaction vs self-interaction}

\author{Kai~Du\thanks{Shanghai Center for Mathematical Sciences, 
		Fudan University, 
		Shanghai 200438, China (email: {\tt kdu@fudan.edu.cn}). 
		K.~Du was partially supported 
		by National Key R{\&}D Program of China (No.~2018YFA0703900), 
		and by Natural Science Foundation of Shanghai (No.~20ZR1403600).
	}
	\and Yifan~Jiang\thanks{Mathematical Institute,
		University of Oxford,
		Oxford OX2 6GG,
		United Kingdom (email: {\tt yifan.jiang@maths.ox.ac.uk}).
		Y.~Jiang was supported by the EPSRC Centre for Doctoral Training in Mathematics of Random Systems: Analysis, Modelling, and Simulation (EP/S023925/1).
	}
	\and Jinfeng~Li\thanks{School of Mathematical Sciences, 
		Fudan University, 
		Shanghai 200433, China (email: {\tt lijinfeng@fudan.edu.cn}).}
	}
\date{}

\begin{document}
\maketitle


\begin{abstract}
This paper proves that, under a monotonicity condition, the invariant probability measure of a McKean--Vlasov process can be approximated by weighted empirical measures of some processes including itself. 
These processes are described by distribution dependent or empirical measure dependent stochastic differential equations constructed from the equation for the McKean--Vlasov process. 
Convergence of empirical measures is characterized by upper bound estimates for their Wasserstein distance to the invariant measure. 
The theoretical results are demonstrated via a mean-field Ornstein--Uhlenbeck process.
\medskip

\medskip

\end{abstract}

\section{Introduction}
The invariant measure and its computation are important topics in the theory and applications of stochastic processes.
Birkhoff's ergodic theorem reveals a link between 
invariant measures and empirical measures of Markov processes 
(cf.~\cite{birkhoff1931proof,da1996ergodicity});
for instance, if $Z$ is a Markov process with a unique invariant (probability) measure $\mu$,
then the ergodic theorem implies that the empirical measure
\begin{equation}
	\label{eqn-emp}
	\cE_{t}[Z]:=\frac{1}{t}\int_{0}^{t}\delta_{Z_{r}}\md r
	=\int_0^1\delta_{Z_{ts}}\md s,
	\quad t > 0
\end{equation}
converges weakly to $\mu$ as $t\to\infty$, 
where $\delta_z$ denotes the Dirac measure at $z$.
Such a property is also called the law of large numbers for empirical measures (cf.~\cite{gine1984some}),
resulting in a practical way to approximate invariant measures.

This paper considers a similar problem for McKean--Vlasov processes: whether and how the invariant measure of a McKean--Vlasov process can be approximated by the empirical measures
of the process itself or some other processes.
A McKean--Vlasov process is a stochastic process described by a stochastic differential equation (SDE) whose coefficients depend on the distribution of the solution,
so it is not a Markov process in the usual sense. 
The equation is often referred as the McKean--Vlasov SDE or the mean-field SDE, established from the ``propagation
of chaos'' of particle systems with mean-field interaction (cf.~\cite{kac1956foundations,mckean1966class})
and applied broadly in sciences, economics, neuroscience, and machine learning; see, for example,~\cite{sznitman1991topics,lasry2007mean,touboul2014prop,
mehri2020propagation}, and the monograph~\cite{carmona2018probabilistic}.
An extensive literature review on the study and applications of McKean--Vlasov SDEs can be found in a recent paper~\cite{sharrock2021parameter}.

Let \(W\) be a standard \(d\)-dimensional Wiener process defined on  
a complete filtered probability space \((\Omega,\cF,\cF_t,\bP)\).
We consider a McKean--Vlasov process $X$ on \(\bR^{d}\) governed by the following equation
\begin{equation}
	\label{eqn-mvsde}
	\md X_{t}=b(X_{t},\cL_{t}[X])\md t+\sigma(X_{t},\cL_{t}[X])\md W_{t},
\end{equation}
where \(\cL_{t}[X]\) denotes the law of \(X_t\).
The well-posedness of~\eqref{eqn-mvsde} can be found in~\cite{buckdahn2017mean, wang2018distribution, mishura2020existence}.
In contrast to Markov processes, $X$ induces a \emph{nonlinear} semigroup $P^*_t$ on the space of probability measures, i.e., $P^*_t \mu := \cL_t[X]$ with $\cL_0[X] = \mu$.
A probability measure $\mu$ is \emph{invariant} for $X$ or~\eqref{eqn-mvsde} if $P^*_t \mu = \mu$ for all $t>0$.

It can be proved as in~\cite{wang2018distribution} that,
under the following assumption, $X$ has a unique invariant measure, and $\cL_t[X]$ converges to the invariant measure exponentially in Wasserstein distance. In what follows,
$\|\cdot\|$ denotes the Frobenius norm of a matrix; 
$\cP$ denotes the space of probability measures on \(\bR^{d}\) equipped with $2$-Wasserstein distance $\cW_{2}(\cdot,\cdot)$,
and $\cP_r$ contains all $\mu\in\cP$
with $\mu(|\cdot|^r):=\int |x|^r \,\mu(\md x)<\infty$.

\begin{asmp}
\label{assump}
The coefficients $b$ and $\sigma$ are continuous functions on $\bR^d \times \cP_2$  and bounded on any bounded set; there are
constants $\alpha > \beta \ge 0$, $\gamma>0$, $\rho>0$, and $K\ge 0$ such that
\begin{gather}\label{eqn-mono}
	2\la b(x,\mu)-b(y,\nu),x-y \ra + \|\sigma(x,\mu)  -\sigma(y,\nu)\|^{2}
	\leq -\alpha|x-y|^{2}+\beta\cW_{2}(\mu,\nu)^{2},\\
  2\la b(x,\mu), x \ra + (1+\rho) \|\sigma(x,\mu) \|^{2}
	\leq -\gamma|x|^{2} + K[ 1 + \rho + \mu(|\cdot|^2) ]\nonumber
\end{gather}
for all \(x,\,y\in\bR^{d}\) and \(\mu,\,\nu\in \cP_2\).
\end{asmp}

We find that this assumption also leads to the convergence of the empirical measure $\cE_t[X]$,
which allows us to approximate the invariant measure for~\eqref{eqn-mvsde} by observing or simulating a trajectory of $X$.
More specifically, we have the following result.

\begin{thm}
	\label{thm-01}
Let Assumption~\ref{assump} be satisfied, and $\im\in\cP_2$ be the invariant measure for~\eqref{eqn-mvsde}.
Then for any \(\nu\in\cP_{2+\rho}\),
	\begin{equation*}
		\E^{\nu}\big[ \cW_{2}(\cE_{t}[X],\im)^{2} \big]
		= O(t^{-\eps})
	\end{equation*}
	with any $\eps<\frac{\rho}{(d+2)(\rho + 2)}$;
in particular, if $\sigma$ is bounded, then $\eps < \frac{1}{d+2}$.
\end{thm}

On the other hand, simulating a trajectory of $X$ in practice meets a critical issue: how to compute or simulate the distributions involved in the equation?
Some natural ways include using the propagation of chaos (cf.~\cite{veretennikov2006ergodic}) 
or solving the nonlinear Fokker--Planck--Kolmogorov equation (cf.~\cite{bogachev2019convergence}),
but normally, the distributions cannot be computed out explicitly in either way, especially when taking account of the computation capacity and cost.

Intuitively, the fact that $\cL_t[X]$ and $\cE_t[X]$ get closer gradually
suggests a modification of~\eqref{eqn-mvsde}: substituting the distribution argument with the empirical measure of the process.
In other words, we introduce the following path-dependent SDE
\begin{equation}
	\label{eqn-emvsde}
	\md Y_{t}=b(Y_{t},\cE_{t}[Y])\md t+\sigma(Y_{t},\cE_{t}[Y])\md W_{t}.
\end{equation}
Interestingly enough, we have the following convergence result.

\begin{thm}
	\label{thm-02}
Under the setting of Theorem~\ref{thm-01}, it holds that for any \(y\in\bR^d\),
	\begin{equation*}
		\E^{y}\big[ \cW_{2}(\cE_{t}[Y],\im)^{2} \big]
		= O(t^{-\eps})
	\end{equation*}
	with any $\eps< (1-\frac{\beta}{\alpha}) \wedge \frac{\rho}{(d+2)(\rho + 2)}$.
\end{thm}

Simulating the trajectory of $Y$ is quite straightforward and standard because it depends only on the past of itself.
This is a big advantage by using~\eqref{eqn-emvsde} to approximate the invariant measure for~\eqref{eqn-mvsde}.
\smallskip

To enhance the generality and practicality of the results, we extend our investigation to generalized empirical measures and particle systems.

Empirical measures of form \eqref{eqn-emp} may be most commonly used, but not the most convenient or practical.
Actually, delayed and/or discrete sampling is much common in practice;
for instance, the empirical measure of the process $Y$ may have the from
\begin{equation}\label{dicrete}
\frac{1}{n_t} \sum_{k=0}^{n_t-1} \delta_{Y_{(k\tau - \theta)\vee 0}}\quad
\text{with}\ n_t = [t/\tau],
\end{equation}
where $\tau > 0$ is the sampling period and $\theta \ge 0$ is the delay.

To cover this situation and extend as far as possible, 
we generalize \eqref{eqn-emp} to the weighted empirical measure of the following form
\begin{equation}\label{eq-Pi}
\cE^{{\wl}}_t[Z] := \int_0^1 \delta_{Z_{ts}} \,{\wl}_t (\md s)
\quad \text{ for }
{\wl}\in\Pi:=\big\{{\wl} = ({\wl}_t)_{t\ge 0}: {\wl}_t \in\cP[0,1]\big\},
\end{equation}
where $\cP[0,1]$ is the space of probability measures on $[0,1]$.
Clearly, \(\cE_{t}\) define in~\eqref{eqn-emp} can be represented as \(\cE^{{\wl}}_{t}\) with \({\wl}_{t}\) being the Lebesgue measure.

For ${\wl},{\wp} \in\Pi$, we introduce the following equation
\begin{equation*}
		\md Y_{t}=b(Y_{t},\cE^{{\wl}}_{t}[Y])\md t+\sigma(Y_{t},\cE^{{\wl}}_{t}[Y])\md W_{t},
\end{equation*}
and expect the convergence
\begin{equation*}\label{eq-dist}
\E^{y}\big[ \cW_{2}(\cE_{t}^{{\wp}}[Y],\im)^{2} \big]
= O(t^{-\eps}).
\end{equation*}
In the other words, the empirical measures used in the equation and in the approximation can be generalized simultaneously, and differently.
Then our aim is to find suitable conditions satisfied by ${\wl}$ and ${\wp}$ for the desired convergence.

Furthermore, we combine the empirical approximation with the classical multi-particle approximation.
Consider the following particle system driven by independent Brownian motions \(W^{i}\):
\[
	\md Y_{t}^{i}=b\bigg(Y_{t}^{i},\frac{1}{N}\sum_{j=1}^{N}\cE_{t}^{{\wl}}[Y^{j}]\bigg)\md t +\sigma\bigg(Y_{t}^{i},\frac{1}{N}\sum_{j=1}^{N}\cE_{t}^{{\wl}}[Y^{j}]\bigg)\md W_{t}^{i}.
\]
Intuitively, as more particles are involved, the approximation by $\frac{1}{N}\sum_{i=1}^{N}\cE_{t}^{{\wp}}[Y^{i}]$
to the invariant measure $\im$ for~\eqref{eqn-mvsde}
should be more efficient than $\cE_{t}^{{\wp}}[Y]$.
This idea is clarified theoretically in Theorem~\ref{thm-multi} 
and demonstrated numerically in Section 6.
\smallskip

Our approach is based on the estimates of the Wasserstein distance (or Kantorovich--Rubinstein metric) 
between the empirical measure and the invariant measure.
We remark that if $\cL_0[X] = \im$ then $X$ becomes a stationary process and Birkhoff's ergodic theorem may apply to obtain an empirical law of large numbers, but it is very likely that the ergodic theorem does not work in the general case, namely $\cL_0[X] = \nu \neq \im$.
The Wasserstein distance is originated from optimal transport theory (cf.~\cite{villani2009optimal}):
for two probability measures $\mu$ and $\nu$ on $\bR^d$, the $p$-Wasserstein distance between them
is defined as
\[
\cW_{p}(\mu,\nu) := \inf \big\{(\E|\xi-\eta|^p)^{1/p}:\mathrm{law}(\xi)=\mu,\ \mathrm{law}(\eta)=\nu\big\}.
\]
Convergence in Wasserstein distance for empirical measures have been discussed under the setting of i.i.d. random variables in \cite{horowitz1994mean,fournier2015rate,ambrosio2019pde} and references therein, of Markov chains 
in~\cite{boissard2014mean,riekert2021wasserstein},
and of diffusions with symmetric 
and non-degenerate generators in \cite{wang2019limit,wang2020wasserstein}.
To the best of our knowledge, we are the first to extend the results to nonlinear diffusions and even for path-dependent processes. 

A key step in our argument is to compare the distribution/path dependent SDE with the following Markovian SDE 
\[	
	\md \hat{X}_{t}=b(\hat{X}_{t},\im)\md t +\sigma(\hat{X}_{t},\im)\md W_{t},
\]
where $\im$ is the invariant measure for~\eqref{eqn-mvsde}.
The results in \cite{wang2019limit,wang2020wasserstein} do not seem to apply to this equation
since it is neither symmetric nor non-degenerate.
With the help of a density coupling lemma from \cite{horowitz1994mean},
we prove that (see Proposition~\ref{lem-aux}) 
\begin{equation*}\label{eq-sde}
\E^{\nu} [\cW_{2}(\cE_{t}^{\wp}[\hat{X}],\im)^{2}] = O(t^{-\eps}),\quad \nu\in\cP_{2+\rho}
\end{equation*}
under Assumption~\ref{assump} and a proper condition on ${\wp}$.
This is a new result with respect to classical SDEs.
We remark that, compared to the sharp results in~\cite{ambrosio2019pde,wang2019limit},
the convergence rate can probably be optimized.
This is left to further work as it is not essential to the primary goal of this paper.

The well-posedness of equations we are concerned with is ensured by Theorem~\ref{thm-well-posed} 
that embeds the equations into a class of path-distribution dependent SDEs
and makes use of a weaker version of Assumption~\ref{assump}.
We do not assume the growth condition on $b$ or the Lipschitz property of $\sigma$, 
which are quite common in the existing works such as~\cite{wang2018distribution, huang2017nonlinear}.
With the well-posedness in hand, the existence and uniqueness of invariant measures for~\eqref{eqn-mvsde}
can be proved under Assumption~\ref{assump} by a similar argument as in~\cite{wang2018distribution}.
\smallskip

Let us give some more remarks on this work and related literature.
The main contribution of this paper is to build a theoretical foundation for designing online algorithms to evaluate 
invariant measures for a class of McKean--Vlasov processes.
As we mentioned above, an approach based on the propagation of chaos to approximate the invariant measure has been discussed in \cite{veretennikov2006ergodic}.
We stress that our empirical approximation overcomes two shortages of the classical propagation of chaos.
First and foremost, we provide an unbiased approximation as \(t\) goes to infinity,
and even one trajectory is enough for computing the approximating measures;
on the other hand, the propagation of chaos approximates the McKean--Vlasov process by an \(N\)-particle system leading to a sampling bias which will not diminish with time (as $N$ is fixed in computation).
Second, given a desired accuracy, the particle number \(N\) cannot be predetermined,
and one has to recalculate from the initial time if an increase of the particle number is needed;
by contrast, our algorithm allows us to make full use of the history information.
To some extent, we do compute infinite sample paths by admitting the equivalence between the temporal and spatial average.
A similar algorithm has been suggested
in a very recent work \cite{alrachid2019new} but without rigorous discussion.

The empirical measure dependent SDE~\eqref{eqn-emvsde} relates to an important class of continuous-time processes called self-interacting diffusions that were initially used to model the evolution of polymers (see~\cite{durrett1992asymptotic,benaim2002self,pemantle2007survey} for details and historical remarks).
A typical form of equations for self-interacting diffusions is
\begin{equation*}
\md Y_{t}= \sqrt{2} \md W_{t} - \frac{1}{t} \bigg[\int_0^t \nabla V(Y_t, Y_s) \md s\bigg] \md t,
\end{equation*}
where $V$ is a potential function. 
Convergence of empirical measures is a most concerned issue for those processes.
Thanks to the gradient form of the drift term, the limit measure, if it exists, is a fixed point of the following mappings (cf.~\cite{benaim2002self}): 
\[
\nu\mapsto \mathcal{M}(\nu)(\md x) := \frac{\me^{-V(x,\cdot)* \nu}}{\mathcal{Z}(\nu)} \md x
\quad \text{with}\ 
\mathcal{Z}(\nu)= \int_{\bR^d} \me^{-V(x,\cdot)* \nu} \md x .
\]
On the other hand, the fixed points of this mapping also coincide with the invariant measures for the associated McKean--Vlasov SDE (cf.~\cite{carrillo2003kinetic,kleptsyn2012ergodicity}).
In some sense, this work extends the research on self-interacting diffusions non-trivially to a general setting.

Another contribution of this work is to establish a level of equivalence between mean-field interaction and self-interaction,
in the sense that large particle systems with mean-field interaction and with self-interaction may evolve to the same equilibrium state eventually.
This inspires that, in some circumstances, the role of one type of interaction can be replaced by the other;
for example, when the mean-field information is temporarily absent for an agent,
an alternative could be to use its past information instead.
This idea may have potential applications in mean-field multi-agent reinforcement learning (cf.~\cite{yang2018mean,gu2021mean}).

In this paper we focus on the McKean--Vlasov process that has a unique invariant measure, and for this,
we employ a weaker version of the condition used in a recent paper~\cite{wang2018distribution}. 
Actually, earlier results on existence-uniqueness of invariant measures of McKean--Vlasov processes all rely on certain monotonicity conditions; see~\cite{ahmed1993invariant,benachour1998nonlinearI,
veretennikov2006ergodic} for example.
For existence only, some sufficient conditions without monotonicity are proposed in~\cite{ahmed1993invariant,hammersley2021mckean,zhang2021existence,bao2021existence}.
\cite{zhang2021existence}~also obtains some conditions under which McKean--Vlasov processes
have more than one invariant measures.
Long-time convergence of the distributions to the invariant
measure is considered in~\cite{benachour1998nonlinearII,veretennikov2006ergodic,
bogachev2019convergence,eberle2019quantitative} under various settings, 
and the exponential convergence in the Wasserstein distance or the entropy is also
investigated in~\cite{carrillo2003kinetic,
eberle2019quantitative,ren2020exponential,liu2020long,
wang2101exponential} and so on.
To our best knowledge, there is no published work so far addressing the convergence of empirical measures for McKean--Vlasov processes. 

We remark that Assumption~\ref{assump} is by no means optimal for existence and uniqueness of invariant measures for~\eqref{eqn-mvsde}, but is also sharp in the sense that the requirement $\alpha > \beta$ cannot be removed.
For example, the equation 
\[
\md X_t = -(X_t - \E X_t) \md t + \sqrt{2} \md W_t,
\]
satisfies \eqref{eqn-mono} with $\alpha=\beta=1$, but the normal distribution $\mathcal{N}(m, 1)$ with each $m\in\bR$ is an invariant measure for it;
according to Example 2.6 in \cite{zhang2021existence}, the McKean--Vlasov SDE
\[
\md X_{t}=-[X_{t}^{3}-4X_{t}+2\lambda(X_{t}-\mathbb{E}X_{t})]\md t+\md W_{t}
\]
admits more than one invariant measure if $\lambda \ge 6$, while it satisfies \eqref{eqn-mono} with $\alpha = \lambda-4$ and $\beta  = \lambda$, that means
the ratio $\alpha/\beta$ can be arbitrarily close to $1$.
Of course, convergence of empirical measures without uniqueness of invariant measures is another
interesting topic which deserves further study.
\smallskip

The remainder of the paper is organized as follows.
In Section 2, we state the main theorems.
Section 3 obtains the convergence of empirical measures for Markovian SDEs.
In Section 4, we prove Theorems~\ref{thm-wmv} and~\ref{thm-wemv}, the generalized versions of Theorems~\ref{thm-01} and \ref{thm-02}, respectively.
Section 5 discusses the empirical approximation via multi-particle systems.
In Section 6, we demonstrate our theoretical results by a mean-field Ornstein--Uhlenbeck process.
Section 7 is devoted to the well-posedness of a class of path-distribution dependent SDEs that can cover the equations
we considered in this paper.

\section{Main results}
Throughout this paper, Assumption~\ref{assump} is fulfilled and $\im \in \cP_2$ denotes the unique invariant 
measure for~\eqref{eqn-mvsde}. Set
\begin{equation}\label{kappa}
\kappa = \frac{\rho}{(d+2)(\rho+2)}.
\end{equation}

Recall \eqref{eq-Pi} for the definitions of the set $\Pi$ and the weighted empirical measure~$\cE^\pi_t$
with weight family~$\pi\in\Pi$. 
For any \(\eps\in (0, 1]\), we introduce two classes of weight families as follows:
\begin{equation}\label{eqn-Pi}
\begin{aligned}
	\Pi_1(\eps) & =\bigg\{ {\wl}\in\Pi:\limsup_{t\to \infty}\int_0^1 s^{-\eps}\,{\wl}_{t}(\md s)<\frac{\alpha}{\beta}\bigg\}, \\
	\Pi_2(\eps) & =\bigg\{ {\wl}\in\Pi:\limsup_{t\to\infty}\int_0^1 t^{\eps}\wedge s^{-\eps} \,{\wl}_{t}(\md s)<\infty ,\\
	&\qquad\qquad\quad\limsup_{t\to\infty}\int_0^1 \!\!\int_0^1 t^{\eps}\wedge |s_{1}-s_{2}|^{-\eps}
	\,{\wl}_{t}(\md s_{1}){\wl}_{t}(\md s_{2})<\infty\bigg\}.
\end{aligned}
\end{equation}

The following theorems are extensions of Theorems~\ref{thm-01} and \ref{thm-02}, respectively.

\begin{thm}
	\label{thm-wmv}
	Let \({\wl}\in\Pi_1(\eps_{1})\) and \({\wp}\in\Pi_2(\eps_{2})\) with \(\eps_{1},\,\eps_{2}\in(0,1]\),
	and \(X\) satisfy the equation
	\begin{align}
		\label{eqn-wmv}
		\md X_{t} & =b(X_{t},\cL^{{\wl}}_{t}[X])\md t+\sigma(X_{t},\cL^{{\wl}}_{t}[X])\md W_{t}\\
\nonumber	& \text{with}\quad \cL^{{\wl}}_{t}[X]:=\int_0^1 \cL_{ts}[X]\,{\wl}(\md s).
\end{align}
	Then for each \(\nu\in\cP_{2+\rho}\), there exists a constant \(C\) such that
	\begin{equation*}
		\E^{\nu}\big[ \cW_{2}(\cE_{t}^{{\wp}}[X],\im)^{2} \big]\leq C t^{-\eps},
	\end{equation*}
	where $\eps = \eps_{1}\wedge\kappa\eps_{2}$.
\end{thm}

\begin{thm}
	\label{thm-wemv}
	Let \({\wl}\in\Pi_1(\eps_{1}) \cap \Pi_2(\eps_{2})\) and \({\wp}\in\Pi_2(\eps_{2})\) with \(\eps_{1},\,\eps_{2}\in(0,1]\), and \(Y\) satisfy the equation
\begin{equation}\label{eqn-wemv}
\md Y_{t}=b(Y_{t},\cE^{{\wl}}_{t}[Y])\md t+\sigma(Y_{t},\cE^{{\wl}}_{t}[Y])\md W_{t}.
\end{equation}
	Then for each \(y\in {\bR}^{d}\), there exists a constant \(C\) such that
	\begin{equation*}
		\E^{y}\big[ \cW_{2}(\cE_{t}^{{\wp}}[Y],\im)^{2} \big]\leq C t^{-\eps},
	\end{equation*}
	where $\eps = \eps_{1}\wedge\kappa\eps_{2}$.
\end{thm}

\begin{rmk}
Theorem~\ref{thm-well-posed} ensures the well-posedness of \eqref{eqn-wmv} and \eqref{eqn-wemv}, 
as well as the continuity of paths of \(X\) and \(Y\),
which is necessary for the above results.
\end{rmk}

\begin{rmk}
Theorem~\ref{thm-01} follows from Theorem~\ref{thm-wmv} by 
taking $\wl_t \equiv \delta_1 $ and $\wp_t$ to be the Lebesgue measure on $[0,1]$.
In this case, it is easily verified that $\eps_1$ can be $1$ 
and $\eps_2$ can be any positive number less than $1$.
Analogously, Theorem~\ref{thm-02} follows from Theorem~\ref{thm-wemv} with $\wl_t$ and $\wp_t$ both the Lebesgue measure, $\eps_1 < 1-\beta/\alpha$ and $\eps_2<1$.
\end{rmk}

\begin{exam}
Discrete and/or delayed sampling is very common in digital systems.
Let us recall the empirical measure~\eqref{dicrete} and denote it by $\hat{\cE}_{t}$. 
Then we consider
\begin{equation}
	\label{eqn-dis}
	\md \hat{Y}_{t}=b(\hat{Y}_{t}, \hat{\cE}_{t}[\hat{Y}])\md t+\sigma(\hat{Y}_{t}, \hat{\cE}_{t}[\hat{Y}])\md W_{t},
	\quad t>0.
\end{equation}
Notice the fact that \[\hat{\cE}_{t}[\hat{Y}]\equiv \hat{\cE}_{k\tau}[\hat{Y}], \text{  for any }t\in(k\tau,(k+1)\tau].\]
Therefore, on each time interval \((k\tau,(k+1)\tau]\), the equation~\eqref{eqn-dis} does not depend on the instantaneous law.
The so-called ``method of steps'' (cf.~\cite{driver2012ordinary}) can be applied to calculate the state and further update the discrete empirical measure.
Again, we can show that 
\begin{equation*}
	\label{est-dis}
	\E^{x}\big[ \cW_{2}(\hat{\cE}_{t}[\hat{Y}],\im)^{2} \big]=O(t^{-\eps}).
\end{equation*}
Indeed, \(\hat{\cE}_{t}\) defined by~\eqref{dicrete} can be represented as \(\cE^{\hat{\wl}}_{t}\) with \(\hat{{\wl}}_{t}\) defined as
			\[\hat{\wl}_{t}(\cdot)=\frac{1}{n_t}\sum_{k=0}^{n_t-1}\delta_{\frac{k\tau - \theta}{t} \wedge 0}.\]
We point out that \(\hat{{\wl}} = (\hat{{\wl}}_t)\) belongs to 
$\Pi_1(1-\beta/\alpha)$ and $\Pi_2(\eps)$ with any $\eps<1$.
So the desired convergence follows from Theorem~\ref{thm-wemv}. 
\end{exam}

Next, we consider a multi-particle empirical approximation,
in which one can expect a better convergence rate for average empirical measures.

\begin{thm}
	\label{thm-multi}
	Let \({\wl}\in\Pi_1(\eps_{1}) \cap \Pi_2(\eps_{2})\) and \({\wp}\in\Pi_2(\eps_{2})\) with \(\eps_{1},\,\eps_{2}\in(0,1]\), and \(Y^i,\ i=1,\dots,N\) satisfy
	\begin{equation}\label{eq-mpmvsde}
		\md Y_{t}^{i}=b\bigg(Y_{t}^{i},\frac{1}{N}\sum_{j=1}^{N}\cE_{t}^{{\wl}}[Y^{j}]\bigg)\md t +\sigma\bigg(Y_{t}^{i},\frac{1}{N}\sum_{j=1}^{N}\cE_{t}^{{\wl}}[Y^{j}]\bigg)\md W_{t}^{i},
	\end{equation}
	where \(W^{i}\) are independent Wiener processes.
	Then for each \(\bm{y}=(y^1,\dots,y^N)\in\bR^{d\times N}\), there is a constant \(C\) such that
	\begin{equation*}
		\E^{\bm{y}}\biggl[ \cW_{2}\biggl(\frac{1}{N}\sum_{i=1}^{N}\cE_{t}^{{\wp}}[Y^{i}],\im \biggr)^{\!2} \,\biggr] \leq  
		Ct^{-(\eps_{1} \wedge \kappa\eps_{2})} N^{-\kappa} + Ct^{-(\eps_{1}\wedge\eps_{2})}.
	\end{equation*}
\end{thm}
\begin{rmk}
	The estimate can be split into two parts.
	For fixed \(N\), the first term is the leading term when \(t\) goes to infinity. 
	For multi-particle approximation, our result gives the same upper bound of the convergence rate of \(t\) but a constant decreasing with the increase of \(N\).   
	The second term characterizes the convergence rate of 
	\[\E^{\bm{y}}\big[ \cW_{2}(\cL^{{\wp}}_{t}[Y^{i}],\im)^{2} \big],\]
	which is in some sense the limit of 
	\[\E^{\bm{y}}\biggl[ \cW_{2}\biggl(\frac{1}{N}\sum_{i=1}^{N}\cE_{t}^{{\wp}}[Y^{i}],\im \biggr)^{\!2} \,\biggr] \]
for fixed $t$,	 as \(N\) goes to infinity.
\end{rmk}

\section{Convergence of empirical measures for Markovian SDEs}
Recall that Assumption~\ref{assump} is fulfilled
and $\im$ denotes the invariant 
measure for~\eqref{eqn-mvsde}.
Recall the Markovian SDE
\begin{equation}\label{eqn-aux}
	\md \hat{X}_{t}=b(\hat{X}_{t},\im)\md t +\sigma(\hat{X}_{t},\im)\md W_{t}.
\end{equation}
Let \(P_{t}\) be the associated semigroup for \(\hat{X}\),
and \(\hat{X}^{x}\) and \(\hat{X}^{\nu}\) be the solutions of the above equations with initial laws \(\delta_{x}\) and \(\nu\), respectively.
Evidently, $\im$ is the unique invariant measure for $\hat{X}$.

\begin{prop}
	\label{lem-aux}
	Let \({\wp}\in\Pi_2(\eps)\) with \(0<\eps<1\).
	Then for each \(\nu\in\cP_{2+\rho}\), there exists a constant \(C\) such that
	\[\E^{\nu} \big[\cW_{2}(\cE^{{\wp}}_{t}[\hat{X}],\im)^{2}\big]\leq C t^{-\kappa\eps}.\]
\end{prop}

The proof of this result relies on several auxiliary lemmas. 
The following estimates are standard, which can be found in~\cite{wang2018distribution}.

\begin{lem}
	\label{lem-std}
	Under the above setting, \(\E[|\hat{X}^{\nu}_{t}|^{2+\rho}]\) is bounded  uniformly in $t$ for each \(\nu\in\cP_{2+\rho} 
	\); 
	for any $x_1,x_2\in\bR^d$,
	\[\E|\hat{X}^{x_{1}}_{t}-\hat{X}^{x_{2}}_{t}|^{2}\leq C |x_1-x_2|^{2} \me^{ - \alpha t},\]
	where \(C\) is a constant independent of initial states; and for any $\nu\in\cP_2$,
 \[\cW_{2}(\cL_{t}[\hat{X}^{\nu}],\im)^{2}\leq \cW_{2}(\nu,\im)^{2}\me^{ - \alpha t}.\]
\end{lem}

A density coupling lemma due to~\cite{horowitz1994mean}
plays a crucial role in our proof, which gives an upper bound of Wasserstein distance
in terms of the density function of the measures.

\begin{lem}[density coupling lemma]
	\label{lem-decouple}
	Let \(f\) and \(g\) be probability density functions on \(\bR^{d}\) such that
	\[\int_{\bR^{d}} |x|^{2}[f(x)+g(x)]\md x <\infty,\]
	and define \(\mu(\md x)=f(x)\md x\) and \(\nu(\md x)=g(x)\md x\).
	Then one has
	\[\cW_{2}(\mu,\nu)^{2}\leq C\int_{\bR^{d}}|x|^{2}|f(x)-g(x)|\md x,\]
	for some constant \(C\).
\end{lem}

The next result is a perturbation lemma for measures.

\begin{lem}
	\label{lem-moll}
	Let \(\xi\) be a random variable with law \(\Phi\in\cP_2\).
Then
	\[\cW_{2}(\Phi*\mu,\mu)^2\leq \E|\xi|^{2}\]
	for any \(\mu\in\cP_2\).
\end{lem}
\begin{proof}
    Let $\eta$ be a random variable with law $\mu$ and independent of $\xi$.
    Then the law of $\xi + \eta$ is $\Phi*\mu$, so
    \(
    \cW_{2}(\Phi*\mu,\mu)^2 \le \E|(\xi + \eta) - \eta|^2 = \E|\xi|^2
    \).
\end{proof}

We also need a weak version of Poincar\'e's inequality.

\begin{lem}
	\label{lem-grad}
	Let \(\mu
	\in \cP_2\)
	and \(f\in\cC^{1}_{b}(\bR^{d})\) such that $\mu(f) = 0$.
	Then
	\[\mu(f^2) \leq \mathrm{Var}(\mu) \left\|\nabla f \right\|^{2}_{\infty},\]
	where
	\[\mathrm{Var}(\mu):=\int_{\bR^{d}}|x|^{2}\,\mu(\md x)-\biggl|\int_{\bR^{d}}x\,\mu(\md x)\biggr|^{2}.\]
\end{lem}
\begin{proof}
One computes that
	\begin{align*}
		\int_{\bR^{d}} f^{2}(x)\mu(\md x)& = \frac{1}{2}\iint_{\bR^{d}\times\bR^{d}} [f^{2}(x)+f^{2}(y)]\,\mu(\md x)\mu(\md y)   \\
		& =                       \frac{1}{2}\iint_{\bR^{d}\times\bR^{d}} [f(x)-f(y)]^{2}\,\mu(\md x)\mu(\md y)       \\
		& \le                       \left\|\nabla f \right\|^{2}_{\infty}\iint_{\bR^{d}\times\bR^{d}} \frac{|x-y|^{2}}{2}\,\mu(\md x)\mu(\md y)       \\
		& =                     \mathrm{Var}(\mu) \left\|\nabla f \right\|^{2}_{\infty}.
	\end{align*}
	The lemma is proved.
\end{proof}

Now we are in a position to prove~Proposition~\ref{lem-aux}.

\begin{proof}[Proof of Proposition~\ref{lem-aux}]
The first step is to mollify all the laws by convolution with some smooth density functions.
Let \(\xi \) be a bounded random vector in \(\bR^{d}\) with smooth density.
Without loss of generality, we assume \(|\xi|\leq 1\).
Let \(\Phi_{r}\) be the law of \(r \xi\) and \(\phi_{r}\) be its density function.
We introduce the mollified empirical measures
\[\cE^{{\wp}}_{t,r}[\hat{X}]:=\Phi_{r}*\cE^{{\wp}}_{t}[\hat{X}]
=\int_{0}^{1} \big(\Phi_{r} * \delta_{\hat{X}_{ts}} \big)\,{\wp}_{t}(\md s)\]
and the mollified laws
\[\cL^{{\wp}}_{t,r}[\hat{X}]:= \Phi_{r} * \cL^{{\wp}}_{t}[\hat{X}]=\int_{0}^{1}
\big(\Phi_{r} * \cL_{ts}[\hat{X}]\big)\,{\wp}_{t}(\md s).\]
One can see that the density functions of $\cE^{{\wp}}_{t,r}[\hat{X}]$ and $\cL^{{\wp}}_{t,r}[\hat{X}]$ are
\[\int_{0}^{1} \phi_{r}(\cdot-\hat{X}_{ts})\,{\wp}_{t}(\md s)
\quad\text{and}\quad
\int_{0}^{1} \E[\phi_{r}(\cdot-\hat{X}_{ts})]\,{\wp}_{t}(\md s),\]
respectively.

To estimate \(\E^{\nu}[\cW_{2}(\cE^{{\wp}}_{t}[\hat{X}],\im)]\),  we write
\begin{equation}\label{ineq-1}
\begin{aligned}
\cW_{2}(\cE^{{\wp}}_{t}[\hat{X}],\im)
& \le \cW_{2}(\cE^{{\wp}}_{t}[\hat{X}],\cE^{{\wp}}_{t,r}[\hat{X}])
+ \cW_{2}(\cE^{{\wp}}_{t,r}[\hat{X}],\cL^{{\wp}}_{t,r}[\hat{X}])\\
&\quad + \cW_{2}(\cL^{{\wp}}_{t,r}[\hat{X}],\cL^{{\wp}}_{t}[\hat{X}]) 
+ \cW_{2}(\cL^{{\wp}}_{t}[\hat{X}],\im).
\end{aligned}
\end{equation}
By Lemma \ref{lem-moll}, it holds almost surely (a.s.) that
\begin{equation}
	\label{ineq-2}
	\cW_{2}(\cE^{{\wp}}_{t}[\hat{X}],\cE^{{\wp}}_{t,r}[\hat{X}])^{2}
	+ \cW_{2}(\cL^{{\wp}}_{t,r}[\hat{X}],\cL^{{\wp}}_{t}[\hat{X}])^{2}
	\leq\E|r\xi|^{2}\leq C r^{2}.
\end{equation}
By Lemma \ref{lem-std}, we have
\[
\begin{aligned}
	\E^{\nu}\big[  \cW_{2}(\cL^{{\wp}}_{t}[\hat{X}],\im)^{2}\big]
	& \leq \int_{0}^{1}\E^{\nu}\big[  \cW_{2}(\cL_{ts}[\hat{X}],\im)^{2} \big]\, {\wp}_{t}(\md s)\nonumber\\
	& \leq \cW_{2}(\nu,\im)^{2}\int_{0}^{1}\me^{-\alpha ts}\, {\wp}_{t}(\md s).
\end{aligned}
\]
When $t$ is sufficiently large, one has $\me^{-\alpha ts} \le 1\wedge (ts)^{-\eps }$ for all $s \in (0,1)$.
Keeping in mind that $\wp \in \Pi_2(\eps)$, one has
\begin{equation}\label{ineq-3}
\begin{aligned}
	\E^{\nu}\big[  \cW_{2}(\cL^{{\wp}}_{t}[\hat{X}],\im)^{2}\big]
	& \leq C\cW_{2}(\nu,\im)^{2}\int_{0}^{1}1\wedge (ts)^{-\eps }\,{\wp}_{t}(\md s)\\
	& \leq C t^{-\eps }\int_{0}^{1}t^{\eps }\wedge s^{-\eps }\,{\wp}_{t}(\md s)\\
	& \leq C t^{-\eps }.
\end{aligned}
\end{equation}
Hereafter, $C$ is a generic constant and may differ from line to line.

Now, we focus on the estimate of \(\cW_{2}(\cE^{{\wp}}_{t,r}[\hat{X}],\cL^{{\wp}}_{t,r}[\hat{X}])\).
By Lemma \ref{lem-decouple}, we have
\begin{align*}
	&\E^{\nu}\big[\cW_{2}(\cE^{{\wp}}_{t,r}[\hat{X}],\cL^{{\wp}}_{t,r}[\hat{X}])^{2}\big]\\
	\leq\,& C\E^{\nu} \int_{\bR^{d}} |y|^{2}\bigg|\int_{0}^{1}\big[\phi_{r}(y-\hat{X}_{ts})-\E^{\nu}\phi_{r}(y-\hat{X}_{ts})\big]{\wp}_{t}(\md s)\bigg|\md y                     \\
	\leq  \,                               & C\int_{|y|>R} |y|^{2}\,\E^{\nu}\bigg|\int_{0}^{1}\big[\phi_{r}(y-\hat{X}_{ts})-\E^{\nu}\phi_{r}(y-\hat{X}_{ts})\big]{\wp}_{t}(\md s)\bigg|\md y       \\
	                                 & +C\int_{|y|\leq R} |y|^{2}\,\E^{\nu}\bigg|\int_{0}^{1}\big[\phi_{r}(y-\hat{X}_{ts})-\E^{\nu}\phi_{r}(y-\hat{X}_{ts})\big]{\wp}_{t}(\md s)\bigg|\md y\\
	                                 =:\,& J_1 + J_2
\end{align*}
where \(r\) and \(R\) are related to \(t\) and will be specified later.

For $J_1$, it follows from Fubini's theorem that
\begin{align*}
	J_1 &\leq C\int_{|y|>R} |y|^{2}\,\E^{\nu}\bigg[\int_{0}^{1} 
	\big\{\phi_{r}(y-\hat{X}_{ts})+\E^{\nu}\phi_{r}(y-\hat{X}_{ts})\big\}\,{\wp}_{t}(\md s)\bigg] \md y                                                                \\
	& = C\int_{0}^{1}\!\!\int_{|y|>R} \E^{\nu}\big[|y|^{2} \phi_{r}(y-\hat{X}_{ts})\big]\md y{\wp}_{t}(\md s)                                                                       \\
	& \leq  C\int_{0}^{1}\!\!\int_{|y|>R} \E^{\nu}\big[|y-\hat{X}_{ts}|^{2} \phi_{r}(y-\hat{X}_{ts})+|\hat{X}_{ts}|^{2} \phi_{r}(y-\hat{X}_{ts})\big]\md y{\wp}_{t}(\md s)      \\
	& \leq C \int_{0}^{1}\!\!\int_{\bR^{d}}|y|^{2}\phi_{r}(y)\md y {\wp}_{t}(\md s) + C\int_{0}^{1}\!\!\int_{|y|>R}\E^{\nu}\big[ |\hat{X}_{ts}|^{2}\phi_{r}(y-\hat{X}_{ts})\big] \md y  {\wp}_{t}(\md s)                       \\
	& \leq Cr^{2}+C\int_{0}^{1}\!\!\int_{|y|>R}\E^{\nu}\big[ |\hat{X}_{ts}|^{2}\phi_{r}(y-\hat{X}_{ts})\big] \md y  {\wp}_{t}(\md s).
\end{align*}
Since  \(|r\xi|\) is bounded by \(r\), the support of \(\phi_{r}\) is contained in the ball \(B_{r}(0)\).
Therefore, we only need to consider on the event \( \{|\hat{X}_{ts}|>R-r\} \) and we get
\begin{align*}
	&\int_{0}^{1}\!\!\int_{|y|>R}\E^{\nu}\big[ |\hat{X}_{ts}|^{2}\phi_{r}(y-\hat{X}_{ts})\big] \md y  {\wp}_{t}(\md s)\\
	\leq\,&\int_{0}^{1}\E^{\nu}\bigg[ |\hat{X}_{ts}|^{2}\bm{1}_{\{ |\hat{X}_{ts}|>R-r \}}\int_{|y|>R}\phi_{r}(y-\hat{X}_{ts})\md y \bigg] \,{\wp}_{t}(\md s)\\
	\leq\,&\int_{0}^{1}\E^{\nu}\big[ |\hat{X}_{ts}|^{2}\bm{1}_{\{ |\hat{X}_{ts}|>R-r \}}\big]\,{\wp}_{t}(\md s).
\end{align*}
By Lemma \ref{lem-std}, \(\hat{X}_{t}\) is \((2+\rho)\)-integrable.
Applying H\"older's inequality and Markov's inequality, we have
\begin{align*}
	\E^{\nu}\big[ |\hat{X}_{ts}|^{2}\bm{1}_{\{ |\hat{X}_{ts}|>R-r \}}\big]
	&\leq\big[\E^{\nu}|\hat{X}_{ts}|^{\rho+2}\big]^{\frac{2}{\rho+2}}\big[ \E^{\nu} \bm{1}_{\{ |\hat{X}_{ts}|>R-r \}} \big]^{\frac{\rho}{\rho+2}}\\
	&\leq \E^{\nu}|\hat{X}_{ts}|^{\rho+2}(R-r)^{-\rho}\\
	&\leq C(R-r)^{-\rho}.
\end{align*}
Combining the above estimates, we prove that
\begin{equation}
	\label{ineq-parta}
	J_1 \leq Cr^{2}+C(R-r)^{-\rho}.
\end{equation}

To estimate $J_2$, an important observation is to represent $J_2$ by double integration and take advantage of the Markov property.
For the notational simplicity, we denote
\[f(s,x,y):=\phi_{r}(y-x)-\E^{\nu}[\phi_{r}(y-\hat{X}_{s})].\]
Then, we represent $J_2$ as
\begin{equation}	\label{eqn-pb}
\begin{aligned}
	J_2 & = C\int_{|y|\leq R} |y|^{2}\,\E^{\nu}\bigg\{\int_0^1 \!\!\int_0^1 f(ts_{1},\hat{X}_{ts_{1}},y)f(ts_{2},\hat{X}_{ts_{2}},y) \,{\wp}_{t}(\md s_{1}){\wp}_{t}(\md s_{2})\bigg\}^{\frac{1}{2}} \md y\\
	& \leq C\int_{|y|\leq R} |y|^{2}\bigg\{\E^{\nu} \int_0^1 \!\!\int_0^1 f(ts_{1},\hat{X}_{ts_{1}},y)f(ts_{2},\hat{X}_{ts_{2}},y) \,{\wp}_{t}(\md s_{1}){\wp}_{t}(\md s_{2})\bigg\}^{\frac{1}{2}}  \md y\\
	& \leq C\int_{|y|\leq R} |y|^{2}\bigg\{\E^{\nu}\bigg[\int_{0}^{1}f(ts_{1},\hat{X}_{ts_{1}},y){\wp}_{t}(\md s_{1})\int_{s_{1}}^{1} f(ts_{2},\hat{X}_{ts_{2}},y){\wp}_{t}(\md s_{2})\bigg]\bigg\}^{\frac{1}{2}}  \md y.
\end{aligned}
\end{equation}
Using the Markov property of \(\hat{X}\), we have
\begin{align*}
	&\E^{\nu}\bigg[\int_{0}^{1}f(ts_{1},\hat{X}_{ts_{1}},y)\,{\wp}_{t}(\md s_{1})\int_{s_{1}}^{1} f(ts_{2},\hat{X}_{ts_{2}},y)\,{\wp}_{t}(\md s_{2})\bigg] \\
	=\,&\E^{\nu}\bigg\{\int_{0}^{1}f(ts_{1},\hat{X}_{ts_{1}},y)\,{\wp}_{t}(\md s_{1})\,\E^{\nu}\bigg[ \int_{s_{1}}^{1}f(ts_{2},\hat{X}_{ts_{2}},y)\,{\wp}_{t}(\md s_{2})\,\bigg\vert\,\hat{X}_{ts_{1}} \bigg]\bigg\} \\
	=\,&\E^{\nu}\bigg[\int_{0}^{1}f(ts_{1},\hat{X}_{ts_{1}},y)\,{\wp}_{t}(\md s_{1})\int_{s_{1}}^{1} P_{ts_{2}-ts_{1}}f(ts_{1},\hat{X}_{ts_{1}},y)\,{\wp}_{t}(\md s_{2}) \bigg],
\end{align*}
where \(P\) acts on the second argument of \(f\).
Plugging this into \eqref{eqn-pb}, we get
\begin{align*}
	J_2 & \leq C\int_{|y|\leq R} |y|^{2}\bigg\{\E^{\nu}\bigg[\int_{0}^{1}f(ts_{1},\hat{X}_{ts_{1}},y) \,{\wp}_{t}(\md s_{1})\\
	&\qquad\qquad\qquad\qquad\qquad \times \int_{s_{1}}^{1} P_{ts_{2}-ts_{1}}f(ts_{1},\hat{X}_{ts_{1}},y) \,{\wp}_{t}(\md s_{2}) \bigg]\bigg\}^{\frac{1}{2}}  \md y\\
	& \leq C \|f \|_{\infty}^{\frac{1}{2}}\int_{|y|\leq R} |y|^{2}  \bigg\{ \int_{0}^{1}{\wp}_{t}(\md s_{1})\int_{s_{1}}^{1}\E^{\nu}\big|P_{ts_{2}-ts_{1}}f(ts_{1},\hat{X}_{ts_{1}},y)\big|{\wp}_{t}(\md s_{2}) \bigg\}^{\frac{1}{2}}  \md y\\
	& \leq C r^{-\frac{d}{2}} R^{2}\int_{|y|\leq R}  \bigg\{ \int_{0}^{1}{\wp}_{t}(\md s_{1})\int_{s_{1}}^{1}\E^{\nu}\big|P_{ts_{2}-ts_{1}}f(ts_{1},\hat{X}_{ts_{1}},y)\big| {\wp}_{t}(\md s_{2}) \bigg\}^{\frac{1}{2}}  \md y.
\end{align*}
Noticing \(\E^{\nu}[ P_{ts_{2}-ts_{1}}f(ts_{1},\hat{X}_{ts_{1}},y) ]=0\), by Lemma \ref{lem-grad} and Lemma \ref{lem-std} we have
\begin{align*}
	     & \E^{\nu}\big|P_{ts_{2}-ts_{1}}f(ts_{1},\hat{X}_{ts_{1}},y) \big| \\
	\leq\, & \big\{ \E^{\nu}\big[ |P_{ts_{2}-ts_{1}}f(ts_{1},\hat{X}_{ts_{1}},y)|^{2} \big] \big\}^{\frac{1}{2}}\\
	\leq\, & \mathrm{Var}(\cL_{ts_{1}}[\hat{X}^{\nu}])^{\frac{1}{2}}\|\nabla_{\!z} [P_{ts_{2}-ts_{1}} \phi_{r}(y-z)] \|_{\infty}                                                                                                         \\
	=   \, & C\sup_{z\in\bR^{d}}\limsup_{|h|\to 0}\frac{|\E^{z+h}[\phi_{r}(y-\hat{X}_{ts_{2}-ts_{1}})]-\E^{z}[\phi_{r}(y-\hat{X}_{ts_{2}-ts_{1}})]|}{|h|}                             \\
	\leq\, & C \| \nabla \phi_{r}\|_{\infty}\sup_{z\in\bR^{d}}\limsup_{|h|\to 0} \frac{ (\E|\hat{X}^{z+h}_{ts_{2}-ts_{1}}-\hat{X}^{z}_{ts_{2}-ts_{1}}|^{2}\big)^{\frac{1}{2}}}{|h|} \\
	\leq\, & C \| \nabla \phi_{r}\|_{\infty}\exp\Big[-\frac{\alpha t(s_{2}-s_{1})}{2}\Big]                                                                                   \\
	\leq\, & C r^{-d}\exp\Big[-\frac{\alpha t(s_{2}-s_{1})}{2}\Big].
\end{align*}
From the above computation and keeping $\wp\in\Pi_2(\eps)$ in mind, we show that
\begin{equation}	\label{ineq-partb}
\begin{aligned}
	J_2 
	 & \leq CR^{2}r^{-\frac{d}{2}}\int_{|y|<R}\bigg\{ \int_{0}^{1}{\wp}_{t}(\md s_{1})\int_{s_{1}}^{1}C r^{-d}
	 \me^{-\frac{1}{2}\alpha t(s_{2}-s_{1})} \,{\wp}_{t}(\md s_{2})\bigg\}^{\frac{1}{2}}\md y\\
	 & \leq C R^{d+2}r^{-d}\bigg\{ \int_0^1 \!\!\int_0^1
	 \me^{-\frac{1}{2}\alpha t|s_{2}-s_{1}|}\,{\wp}_{t}(\md s_{1}){\wp}_{t}(\md s_{2}) \bigg\}^{\frac{1}{2}}\\
	 & \leq C R^{d+2}r^{-d}t^{-\frac{\eps }{2}}\bigg\{ \int_0^1 \!\!\int_0^1 t^{\eps_2}\wedge |s_{1}-s_{2}|^{-\eps_2}{\wp}_{t}(\md s_{1}){\wp}_{t}(\md s_{2}) \bigg\}^{\frac{1}{2}}\\
	 & \leq C R^{d+2}r^{-d}t^{-\frac{\eps }{2}}.
\end{aligned}
\end{equation}

Combining \eqref{ineq-1}, \eqref{ineq-2}, \eqref{ineq-3}, \eqref{ineq-parta} and \eqref{ineq-partb}, we gain that
	\[
	\E^{\nu} \big[\cW_{2}(\cE^{{\wp}}_{t}[\hat{X}],\im)^{2}\big]\leq C[r^2 + t^{-\eps } + (R-r)^{\rho} + R^{d+2} r^{-d} t^{-\frac{\eps }{2}}].
	\]
Now taking 
\[r=t^{-\frac{\rho}{2(d+2)(\rho+2)}\eps } \quad\text{and}\quad R=t^{\frac{1}{(d+2)(\rho+2)}\eps },\]
we obtain that for each \(\nu \in \cP_{2+\rho}\),
	\[\E^{\nu} \big[\cW_{2}(\cE^{{\wp}}_{t}[\hat{X}],\im)^{2}\big]\leq C t^{-\kappa\eps }.\]
The proof is complete.
\end{proof}

\section{Proofs of Theorems \ref{thm-wmv} and \ref{thm-wemv}}
By Proposition~\ref{lem-aux}, the key point is to estimate
\begin{equation*}
	\E^{\nu}\big[\cW_{2}(\cE^{{\wp}}_{t}[X],\cE^{{\wp}}_{t}[\hat{X}])^{2}\big] 
	\quad\text{and}\quad  
	\E^{y}\big[\cW_{2}(\cE^{{\wp}}_{t}[Y],\cE^{{\wp}}_{t}[\hat{X}])^{2}\big].
\end{equation*}
The following comparison lemma is crucial to bound the Wasserstein distances above.
\begin{lem}
	\label{lem-comp}
	Let \(0<\eps_{1},\,\eps_{2}<1\), \({\wl}\in\Pi_1(\eps_{1})\) and \(M\) be a non-negative constant.
If  a differentiable function \(f:[0,\infty)\to\bR\) satisfies
	\[
		f'(t)\leq -\alpha f(t)+ \beta \int_{0}^{1}f(ts)\,{\wl}_{t}(\md s) + Mt^{-\eps_{2}},
	\]
then for \(\eps=\eps_{1}\wedge \eps_{2}\),
	\[
		\label{eqn-comp}
		f(t)\leq C t^{-\eps}.
	\]
	where the constant $C$ is independent of $f$.
\end{lem}
\begin{proof}
	Define 
	\[g(t):=t^{-\eps},\]
	where \(\eps\leq\eps_{1}\wedge \eps_{2}\).
	Since \({\wl}\in\Pi_1(\eps_{1})\), we have 
	\[\limsup_{t\to\infty}\int_{0}^{1}s^{-\eps_{1}}\,{\wl}_{t}(\md s)<\frac{\alpha}{\beta}.\]
	By direct computation, there exist numbers \(m,\,T>0\) such that for any \(t>T\)
	\[g'(t)\geq -\alpha g(t)+ \beta \int_{0}^{1}g(ts) \,{\wl}_{t}(\md s) + mt^{-\eps_{2}}.\]
	Define
	\[h(t)=[m^{-1}M+T^{\eps}f(T)]g(t)-f(t).\]
	Then, we have 
	\[h'(t)\geq -\alpha h(t)+ \beta \int_{0}^{1}h(ts)\, {\wl}_{t}(\md s)\]
	and \(h(T)\geq 0\).
	This leads to \(h(t)\geq 0\) for \(t>T\) which implies the desired estimate.
\end{proof}

\begin{proof}[Proof of Theorem \ref{thm-wmv}]
	Let $X$ and $\hat{X}$ be the solutions to \eqref{eqn-wmv} and \eqref{eqn-aux}, respectively, with the same initial data $\xi\sim\nu$.
	We first estimate \(\E^{\nu} |X_{t}-\hat{X}_{t}|^{2}\).
	By It\^o's formula and Assumption~\ref{assump}, we have
	\begin{align*}
		\frac{\md}{\md t}\E^{\nu}|X_{t}-\hat{X}_{t}|^{2}\leq &-\alpha\E^{\nu}|X_{t}-\hat{X}_{t}|^{2} + \beta\E^{\nu}\big[\cW_{2}(\cL_{t}^{{\wl}}[X],\im)^{2}\big]   \\
		\leq& -\alpha\E^{\nu}|X_{t}-\hat{X}_{t}|^{2}+\hat{\beta}\E^{\nu}\big[\cW_{2}(\cL^{{\wl}}_{t}[X],\cL^{{\wl}}_{t}[\hat{X}])^{2}\big]\\
		 &+C\E^{\nu}\big[\cW_{2}(\cL^{{\wl}}_{t}[\hat{X}],\im)^{2}\big]\\
		\leq& -\alpha\E^{\nu}|X_{t}-\hat{X}_{t}|^{2}+\hat{\beta} \int_{0}^{1}\E^{\nu}|X_{ts}-\hat{X}_{ts}|^{2}{\wl}_{t}(\md s)\\
		&+C\E^{\nu}\big[\cW_{2}(\cL^{{\wl}}_{t}[\hat{X}],\im)^{2}\big],
	\end{align*}
	where \(\beta<\hat{\beta}<\alpha\).
	By Lemma \ref{lem-std} and the condition $\wl\in \Pi_1(\eps_1)$, we have
	\begin{align*}
		\E^{\nu}\big[\cW_{2}(\cL^{{\wl}}_{t}[\hat{X}],\im)^{2}\big]
		& \leq \int_{0}^{1}\E^{\nu}\big[ \cW_{2}(\cL^{{\wl}}_{ts}[\hat{X}],\im)^{2} \big]{\wl}_{t}(\md s)\\
		& \le C \int_{0}^{1} \me^{-\alpha ts}\,{\wl}_{t}(\md s)\\
		& \leq C t^{-\eps_{1}}.
	\end{align*}
	Let \(f(t)=\E^{\nu} |X_{t}-\hat{X}_{t}|^{2}\).
	So, we have shown 
	\[f'(t)\leq -\alpha f(t) + \hat{\beta}\int_{0}^{1}f(ts)\,{\wl}_{t}(\md s)+ C t^{-\eps_{1}}.\]
	By Lemma \ref{lem-comp}, this leads to
	\begin{equation}\E^{\nu} |X_{t}-\hat{X}_{t}|^{2}\leq C t^{-\eps_{1}}.\end{equation}
	Then, we gain that
	\begin{align}
		\E^{\nu}\big[\cW_{2}(\cE^{{\wp}}_{t}[X],\cE^{{\wp}}_{t}[\hat{X}])^{2}\big]&\leq
		\int_{0}^{1}\E^{\nu}|X_{ts}-\hat{X}_{ts}|^{2}\,{\wp}_{t}(\md s)\nonumber\\
		&\leq C t^{-\eps}\int_{0}^{1}s^{-\eps}\,{\wp}_{t}(\md s)\nonumber\\
		&\leq C t^{-\eps},
	\end{align}
	where \(\eps=\eps_{1}\wedge\eps_{2}\). 
	Combining with the estimate of \(\E^{\nu}[ \cW_{2}(\cE^{{\wp}}_{t}[\hat{X}],\im)^{2} ]\) in Proposition~\ref{lem-aux}, we finally prove that for each \(\nu\in\cP_{2+\rho} \), there exists a constant \(C\) such that
	\begin{equation*}
		\E^{\nu}\big[ \cW_{2}(\cE_{t}^{{\wp}}[X],\im)^{2} \big]\leq C t^{-(\eps_{1}\wedge\kappa\eps_{2})}.
	\end{equation*}
	The proof is complete.
\end{proof}

\begin{proof}[Proof of Theorem \ref{thm-wemv}]
	The proof is similar to the proof of Theorem \ref{thm-wmv} except that we estimate \(\E^{y}|Y_{t}-\hat{X}_{t}|^{2}\) instead of \(\E^{\nu}|X_{t}-\hat{X}_{t}|^{2}\).
	Let $Y$ and $\hat{X}$ be the solutions to \eqref{eqn-wemv} and \eqref{eqn-aux}, respectively, with the same initial point $y\in\bR^d$.
	By It\^o's formula and Assumption~\ref{assump}, we have
	\begin{align*}
		\frac{\md}{\md t}\E^{y}|Y_{t}-\hat{X}_{t}|^{2} & \leq -\alpha\E^{y}|Y_{t}-\hat{X}_{t}|^{2} + \beta\E^{y}\big[\cW_{2}(\cE_{t}^{{\wl}}[Y],\im)^{2}\big]   \\
		&\leq -\alpha\E^{y}|Y_{t}-\hat{X}_{t}|^{2}+\hat{\beta}\E^{y}\big[\cW_{2}(\cE^{{\wl}}_{t}[Y],\cE^{{\wl}}_{t}[\hat{X}])^{2}\big]\\
		 & \quad +C\E^{y}\big[\cW_{2}(\cE^{{\wl}}_{t}[\hat{X}],\im)^{2}\big]\\
		& \leq -\alpha\E^{y}|Y_{t}-\hat{X}_{t}|^{2}+\hat{\beta} \int_{0}^{1}\E^{y}|Y_{ts}-\hat{X}_{ts}|^{2}\,{\wl}_{t}(\md s)\\
		&\quad +C\E^{y}\big[\cW_{2}(\cE^{{\wl}}_{t}[\hat{X}],\im)^{2}\big],
	\end{align*}
	where \(\beta<\hat{\beta}<\alpha\).
	Since \({\wl}\in \Pi_2(\eps_2)\), by Proposition~\ref{lem-aux} we have 
	\begin{align*}
		\E^{y}\big[\cW_{2}(\cE^{{\wl}}_{t}[\hat{X}],\im)^{2}\big]\leq C t^{-\kappa\eps_2}.
	\end{align*}
	Let \(f(t)=\E^{y} |Y_{t}-\hat{X}_{t}|^{2}\).
	Thus, we obtain that
	\[f'(t)\leq -\alpha f(t) + \hat{\beta}\int_{0}^{1}f(ts) \,{\wl}_{t}(\md s)+ C t^{-\kappa\eps_2}.\]
	By Lemma \ref{lem-comp}, this leads to
	\begin{equation}\E^{y} |Y_{t}-\hat{X}_{t}|^{2}\leq C t^{-(\eps_1 \wedge\kappa\eps_2)}.
	\end{equation}	
	Then we gain that
	\begin{align}
		\E^{y}\big[\cW_{2}(\cE^{{\wp}}_{t}[Y],\cE^{{\wp}}_{t}[\hat{X}])^{2}\big] & \leq
		\int_{0}^{1}\E^{y}|Y_{ts}-\hat{X}_{ts}|^{2}\,{\wp}_{t}(\md s)\nonumber\\
		& \leq C t^{-\eps}\int_{0}^{1}s^{-\eps} \,{\wp}_{t}(\md s)\nonumber\\
		& \leq C t^{-\eps},
	\end{align}
	where \(\eps=\eps_1 \wedge\kappa\eps_2 \). 
	Together with the estimate of \(\E^{y}[ \cW_{2}(\cE^{{\wp}}_{t}[\hat{X}],\im)^{2} ]\) in Proposition~\ref{lem-aux}, we finally obtain that
	\begin{equation*}
		\E^{y}\big[ \cW_{2}(\cE_{t}^{{\wp}}[Y],\im)^{2} \big]\leq C t^{-(\eps_1 \wedge\kappa\eps_2)}.
	\end{equation*}
	The proof is complete.
\end{proof}

\section{Multi-particle empirical approximation}
In this section, we extend our empirical approximation to multi-particle systems,
and prove Theorem~\ref{thm-multi}.
Recall that
\[
	\md Y_{t}^{i}=b\bigg(Y_{t}^{i},\frac{1}{N}\sum_{j=1}^{N}\cE_{t}^{{\wl}}[Y^{j}]\bigg)\md t +\sigma\bigg(Y_{t}^{i},\frac{1}{N}\sum_{j=1}^{N}\cE_{t}^{{\wl}}[Y^{j}]\bigg)\md W_{t}^{i}.
\]

For $i=1,\dots,N$, let
\begin{equation}
	\md \hat{X}_{t}^{i}=b(\hat{X}_{t}^{i},\im)\md t +\sigma(\hat{X}_{t}^{i},\im)\md W_{t}^{i}.
\end{equation}
We first show a multi-particle version of Lemma  \ref{lem-aux}.
\begin{lem}
	\label{lem-multi-aux}
Let \({\wp}\in\Pi_2(\eps )\) with \(0<\eps <1\).
	For each \(\bm{x}=(x^1,\dots,x^N)\in\bR^{d\times N}\), there exists a constant \(C\) such that
	\begin{align*}
	\E^{\bm{x}} \bigg[\cW_{2}\bigg(\frac{1}{N}\sum_{i=1}^{N}\cE^{{\wp}}_{t}[\hat{X}^{i}],\cL^{{\wp}}_{t}[\hat{X}]\bigg)^{\!2}\bigg]
	& \leq C  t^{-\kappa\eps }N^{-\kappa},\\
	\E^{\bm{x}}\big[  \cW_{2}(\cL^{{\wp}}_{t}[\hat{X}],\im)^{2}\big]
	& \leq C t^{-\eps }.
	\end{align*}
	Therefore, one has that
	\[\E^{\bm{x}} \bigg[\cW_{2}\bigg(\frac{1}{N}\sum_{i=1}^{N}\cE^{{\wp}}_{t}[\hat{X}^{i}],\im\bigg)^{\!2}\bigg]\leq C  t^{-\kappa\eps }N^{-\kappa}+C t^{-\eps }.\]
\end{lem}

\begin{proof}
	We follow a similar procedure to the proof of Proposition~\ref{lem-aux}.
	For the mollified empirical measures, we get exactly the same estimates
	\begin{align*}
		\cW_{2}\bigg( \frac{1}{N}\sum_{i=1}^{N}\cE^{{\wp}}_{t}[\hat{X}^{i}],\frac{1}{N}\sum_{i=1}^{N}\cE^{{\wp}}_{t,r}[\hat{X}^{i}] \bigg)^{2}
		+	\cW_{2}(\cL^{{\wp}}_{t,r}[\hat{X}],\cL^{{\wp}}_{t}[\hat{X}])^{2}
		 \leq Cr^{2}.
	\end{align*}

	Now, we focus on the estimate of 
	\[
	\cW_{2}\bigg( \frac{1}{N}\sum_{i=1}^{N}\cE^{{\wp}}_{t}[\hat{X}^{i}],\cL^{{\wp}}_{t,r}[\hat{X}] \bigg).
	\]
	Notice the density of \(\frac{1}{N}\sum_{i=1}^{N}\cE^{{\wp}}_{t}[\hat{X}^{i}]\) is
	\[\int_{0}^{1}\frac{1}{N}\sum_{i=1}^{N}\phi_{r}(\cdot -\hat{X}_{ts}^{i}) \, {\wp}_{t}(\md s).\]
	Therefore, by Lemma \ref{lem-decouple} we have
\begin{align*}
	&\E^{\bm{x}}\bigg[\cW_{2}\bigg(\frac{1}{N}\sum_{i=1}^{N}\cE^{{\wp}}_{t,r}[\hat{X}^{i}],\cL^{{\wp}}_{t,r}[\hat{X}]\bigg)^{\!2}\bigg]\\
	& \leq  C\int_{|y|>R} |y|^{2}\,\E^{\bm{x}}\bigg|\int_{0}^{1}\bigg[\frac{1}{N}\sum_{i=1}^{N}\phi_{r}(y-\hat{X}^{i}_{ts})-\E^{\bm{x}}\phi_{r}(y-\hat{X}_{ts})\bigg]{\wp}_{t}(\md s)\bigg|\md y       \\
	& \quad + C\int_{|y|\leq R} |y|^{2}\,\E^{\bm{x}}\bigg|\int_{0}^{1}\bigg[\frac{1}{N}\sum_{i=1}^{N}\phi_{r}(y-\hat{X}^{i}_{ts})-\E^{\bm{x}}\phi_{r}(y-\hat{X}_{ts})\bigg]{\wp}_{t}(\md s)\bigg|\md y \\
	& =: J_1 + J_2.
\end{align*}
    As  we did in the proof of Proposition~\ref{lem-aux}, we also have
	\[
		J_1\leq Cr^{2}+C (R-r)^{-\rho}.
	\]
	
	For $J_2$, we notice
	\begin{align*}
		J_2 & = C\int_{|y|\leq R} |y|^{2}\,\E^{\bm{x}}\bigg|\int_{0}^{1}\bigg[\frac{1}{N}\sum_{i=1}^{N}\phi_{r}(y-\hat{X}^{i}_{ts})-\E^{\bm{x}}\phi_{r}(y-\hat{X}_{ts})\bigg]{\wp}_{t}(\md s)\bigg|\md y\\
		&= C\int_{|y|\leq R} |y|^{2}\,\E^{\bm{x}}\bigg|\frac{1}{N}\sum_{i=1}^{N}\int_{0}^{1}\big[\phi_{r}(y-\hat{X}^{i}_{ts})-\E^{\bm{x}}\phi_{r}(y-\hat{X}_{ts})\big]{\wp}_{t}(\md s)\bigg|\md y\\
		&\leq C\int_{|y|\leq R} |y|^{2}\bigg\{\E^{\bm{x}}\bigg|\frac{1}{N}\sum_{i=1}^{N}\int_{0}^{1}\big[\phi_{r}(y-\hat{X}^{i}_{ts})-\E^{\bm{x}}\phi_{r}(y-\hat{X}_{ts})\big]{\wp}_{t}(\md s)\bigg|^{2}\bigg\}^{\frac{1}{2}}\md y\\
		&= CN^{-\frac{1}{2}}\int_{|y|\leq R} |y|^{2}\bigg\{\E^{\bm{x}}\bigg|\int_{0}^{1}\big[\phi_{r}(y-\hat{X}^{i}_{ts})-\E^{\bm{x}}\phi_{r}(y-\hat{X}_{ts})\big]{\wp}_{t}(\md s)\bigg|^{2}\bigg\}^{\frac{1}{2}}\md y\\
		&\leq C R^{d+2}r^{-d}t^{-\frac{\eps }{2}}N^{-\frac{1}{2}}.
	\end{align*}
	Here, the third equality comes from the fact that \((\hat{X}^{i})_{1\leq i\leq N}\) are independent and 
	\[\E^{\bm{x}}\int_{0}^{1}\big[\phi_{r}(y-\hat{X}^{i}_{ts})-\E^{\bm{x}}\phi_{r}(y-\hat{X}_{ts})\big]{\wp}_{t}(\md s)=0.\]

	To sum up, we obtain that 
	\[
	\E^{\bm{x}} \bigg[\cW_{2}\bigg(\frac{1}{N}\sum_{i=1}^{N}\cE^{{\wp}}_{t}[\hat{X}^{i}],\cL^{{\wp}}_{t}[\hat{X}]\bigg)^{\!2}\bigg]
	\le C r^2 + C(R-r)^{-\rho} + C R^{d+2}r^{-d}t^{-\frac{\eps }{2}} N^{-\frac{1}{2}}. 
	\]
By setting \[r=t^{-\frac{\rho}{2(d+2)(\rho+2)}\eps }N^{-\frac{\rho}{2(d+2)(\rho+2)}}  \quad\text{and}\quad R=t^{\frac{1}{(d+2)(\rho+2)}\eps }N^{\frac{1}{(d+2)(\rho+2)}},\]
we have that
\[
\E^{\bm{x}} \bigg[\cW_{2}\bigg(\frac{1}{N}\sum_{i=1}^{N}\cE^{{\wp}}_{t}[\hat{X}^{i}],\cL^{{\wp}}_{t}[\hat{X}]\bigg)^{\!2}\bigg]\leq C  t^{-\kappa\eps }N^{-\kappa}.
\]
The proof is complete.
\end{proof}

\begin{proof}[Proof of Theorem \ref{thm-multi}]
	We follow the same lines of the proof of Theorem \ref{thm-wemv}.
	Let \(f(t)=\E^{\bm{y}}|Y^{i}_{t}-\hat{X}^{i}_{t}|^{2}\). 
	By It\^o's formula and Lemma \ref{lem-multi-aux}, we have
	\begin{align*}f'(t) & \leq \alpha f(t)+\hat{\beta}\int_{0}^{1} f(ts)\,{\wl}_{t}(\md s)+C\E^{\bm{y}}\bigg[\cW_{2}\bigg( \frac{1}{N}\sum_{i=1}^{N}\cE^{{\wl}}_{t}[\hat{X}^{i}],\im \bigg)^{\!2}\bigg]\\
		& \leq \alpha f(t)+\hat{\beta}\int_{0}^{1} f(ts)\,{\wl}_{t}(\md s)+C t^{-\kappa\eps_2}N^{-\kappa}+C t^{-\eps_2}
	\end{align*}
	By Lemma \ref{lem-comp}, we obtain
	\begin{equation}
	\label{est-multi}f(t)\leq  C t^{-(\eps_1\wedge\kappa\eps_2)}N^{-\kappa}+C t^{-(\eps_1\wedge\eps_2)}.
	\end{equation}
	Then one has that
	\begin{align*}
			&\E^{\bm{y}}\bigg[ \cW_{2}\bigg(\frac{1}{N}\sum_{i=1}^{N}\cE_{t}^{{\wp}}[Y^{i}],\frac{1}{N}\sum_{i=1}^{N}\cE_{t}^{{\wp}}[\hat{X}^{i}]\bigg)^{\!2} \bigg]\\
			\leq\,&\frac{1}{N}\sum_{i=1}^{N}	\E^{\bm{y}}\big[ \cW_{2}(\cE_{t}^{{\wp}}[Y^{i}],\cE_{t}^{{\wp}}[\hat{X}^{i}])^{\!2} \big]\\
			\leq\,& \int_{0}^{1}\E^{\bm{y}}|Y^{i}_{ts}-\hat{X}^{i}_{ts}|^{2}\,{\wp}_{t}(\md s).
	\end{align*}
	This along with \eqref{est-multi}, we have
		\begin{equation*}
			\E^{\bm{y}}\bigg[ \cW_{2}\bigg(\frac{1}{N}\sum_{i=1}^{N}\cE_{t}^{{\wp}}[Y^{i}],\frac{1}{N}\sum_{i=1}^{N}\cE_{t}^{{\wp}}[\hat{X}^{i}]\bigg)^{\!2} \bigg]
			\leq C t^{-(\eps_1\wedge\kappa\eps_2)}N^{-\kappa}+C t^{-(\eps_1\wedge\eps_2)}.
		\end{equation*}
	This combined with Lemma \ref{lem-multi-aux} concludes the proof.
\end{proof}

\section{Numerical examples}
In this section, we consider the following mean-field Ornstein--Uhlenbeck process:
\begin{equation}
    \label{eqn-ou}
    \md X_{t}=\big(-2X_{t}-\E X_{t}\big)\md t+\big(2-\sqrt{\E|X_{t}|^{2}}\big) \md W_{t}.
\end{equation}
It is straightforward to check that \(X_{t}\) satisfies Assumption~\ref{assump}.
Indeed, \(X_{t}\) has a unique invariant measure \(\im\) following the normal distribution \(\cN(0,4/9)\).

We set the time step \(\tau=1\).
As stated in \eqref{eqn-dis}, we compute  \(\hat{\cE}_{t}[\hat{Y}]\) and compare with \(\hat{\cE}_{t}[X]\).
Notice that \(\hat{Y}\) satisfies
\begin{equation*}
	\label{eqn-e}
		\begin{aligned}
			\md \hat{Y}_{t}&=(-2 \hat{Y}_{t}- m_{t})\md t +(2-\sqrt{v_{t}}) \md W_{t},
		\end{aligned}
\end{equation*}
where 
\begin{equation}
	\label{eqn-m}
	m_{t}=\frac{1}{[ t]}\sum_{k=0}^{[ t] -1} \hat{Y}_{k}\quad\text{ and }\quad v_{t}=\frac{1}{[ t]}\sum_{k=0}^{[ t] -1} |\hat{Y}_{k}|^{2} .
\end{equation}
Since \(m_{t}\) and \(v_{t}\) are piecewise constant, we can compute \(\hat{Y}\) step by step.
On an integer time interval \((k,k+1]\), \(\hat{Y}_{t}\) satisfies 
\[\hat{Y}_{t}=\me^{-2(t-k)}   \hat{Y}_{k}-m_
{k}\int_{k}^{t}\me^{-2(t-s)}\md s +(2-\sqrt{v_{k}})\int_{k}^{t}\me^{-2(t-s)}\md W_{s}.\]
In particular, we obtain a recursive formula:
\begin{equation*}
	\hat{Y}_{k+1}=\me^{-2}   \hat{Y}_{k}-\frac{1-\me^{-2}}{2}m_{k}+\Big(1-\frac{\sqrt{v_{k}}}{2}\Big)\sqrt{1-\me^{-4}}\,\xi_{k},
\end{equation*}
where \(\xi_{k}\) are independent standard normal random variables.
This allows us to compute \(\hat{Y}_{t}\) and its empirical measure online.

Also, we calculate the multi-particle approximation in the same way.
For an \(N\)-particle system, we have
\begin{equation*}
	\hat{Y}^{i}_{k+1}=\me^{-2}   \hat{Y}^{i}_{k}-\frac{1-\me^{-2}}{2}m_{k}+\Big(1-\frac{\sqrt{v_{k}}}{2}\Big)\sqrt{1-\me^{-4}}\,\xi^{i}_{k},
\end{equation*}
where
\begin{equation*}
	m_{t}=\frac{1}{[ t]}\sum_{k=0}^{[ t] -1} \frac{1}{N}\sum_{i=1}^{N}\hat{Y}^{i}_{k}\quad \text{and}\quad v_{t}=\frac{1}{[ t]}\sum_{k=0}^{[ t] -1} \frac{1}{N}\sum_{i=1}^{N}|\hat{Y}^{i}_{k}|^{2}.
\end{equation*}

\begin{figure}[!tb]
	\centering
	\includegraphics[width=\textwidth]{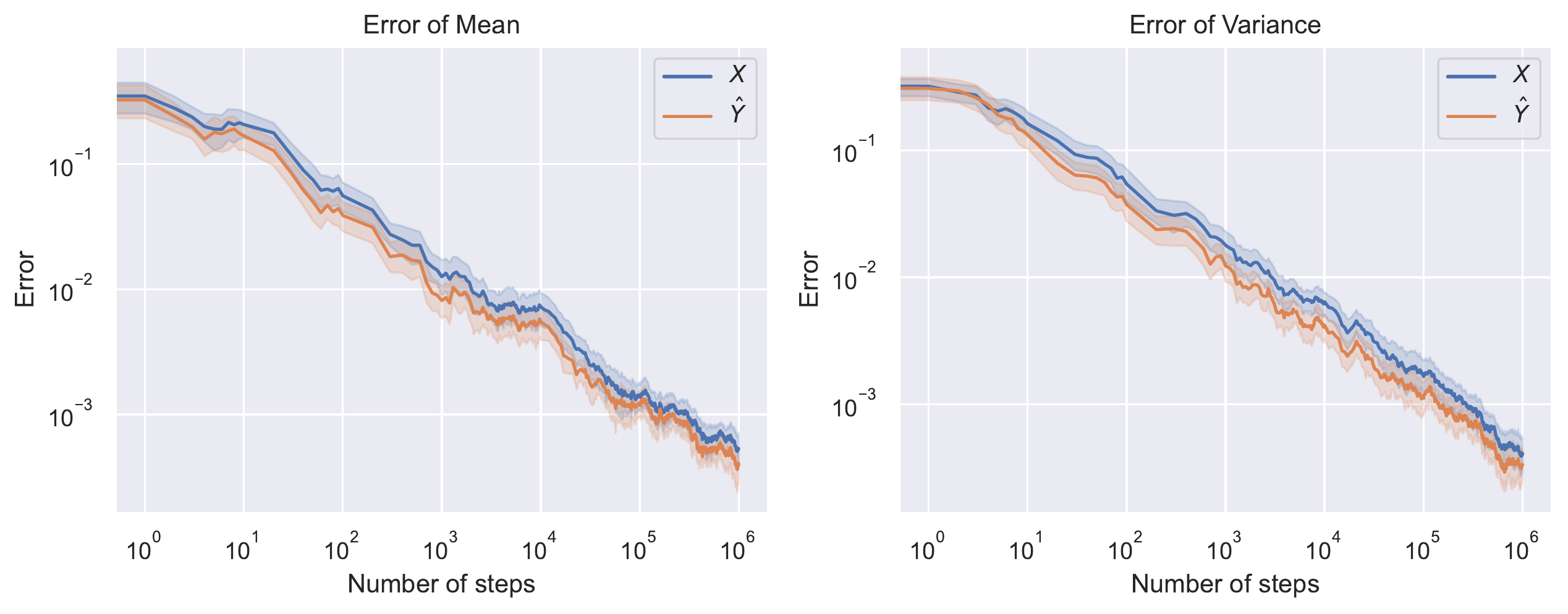}
	\caption{Errors of empirical measures of McKean--Vlasov process \(X\) (with mean-field interaction) and path-dependent process \(\hat{Y}\) (with self-interaction)}
	\label{fig-1}
\end{figure}

\begin{figure}[!tb]
        \centering
        \includegraphics[width=\textwidth]{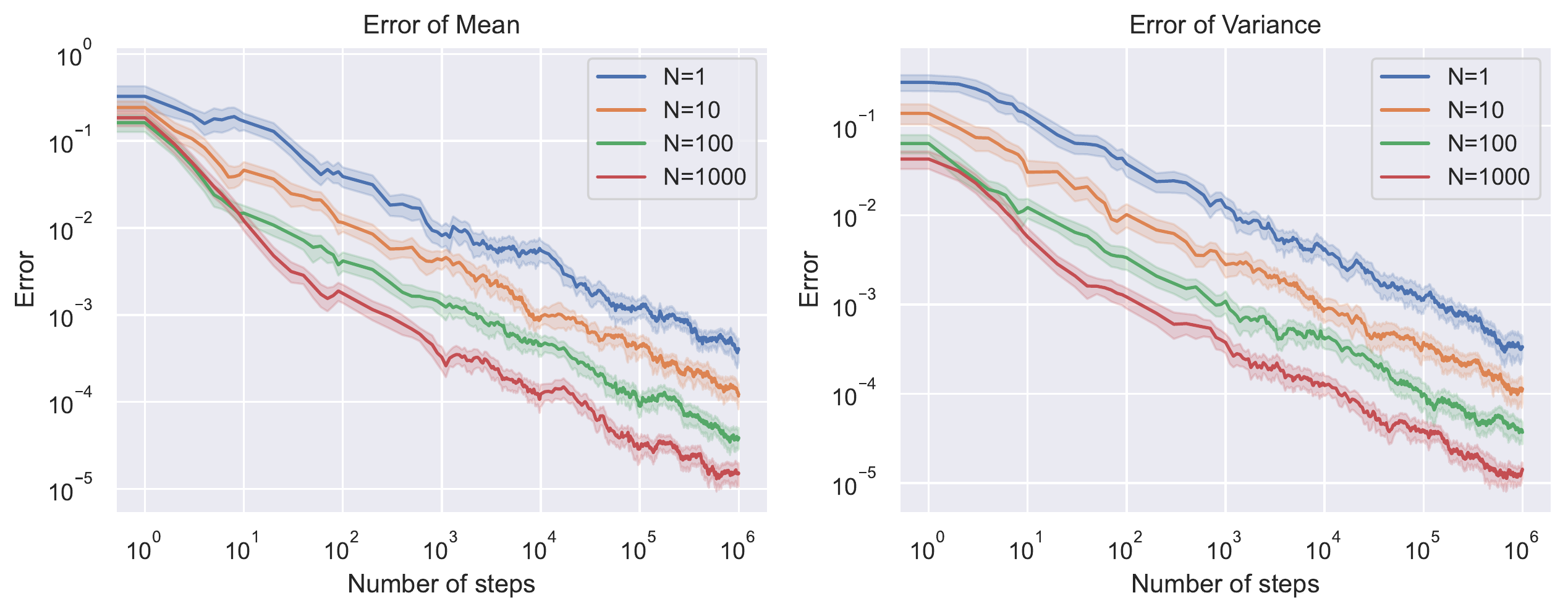}
        \caption{Errors of empirical measures of path-dependent process \(\hat{Y}\) with different number of particles.}
        \label{fig-ou}
\end{figure}

In our numerical experiments, we run \(10^{6}\) steps on a single sample path independently 30 times.
We record the mean and variance of the empirical measures \(\hat{\cE}_{t}[X]\) and \(\hat{\cE}_{t}[\hat{Y}]\).
In Figure~\ref{fig-1}, we illustrate the error of mean (left) and the error of variance (right) against the number of steps.
The shaded areas denote the 95\% confidence interval.
We see that the empirical measures of path-dependent process \(\hat{Y}\) have the same convergence performance as the counterpart of McKean--Vlasov process \(X\).
In Figure~\ref{fig-ou}, we further compare the results of multi-particle systems with 1, 10, 100, 1000 particles.
They roughly have the same convergence rate, while the constant is decreasing with the increase of the  number of particles.

We emphasize that our algorithm is space-saving for a wide range of processes.
In practice, we use
\begin{equation*}
	m_{k+1}=\frac{k}{k+1}m_{k}+\frac{1}{(k+1)N}\sum_{i=1}^{N}Y^{i}_{k} \quad\text{and}\quad v_{k+1}=\frac{k}{k+1}v_{k}+\frac{1}{(k+1)N}\sum_{i=1}^{N}|Y^{i}_{k}|^{2}
\end{equation*}
rather than equation \eqref{eqn-m}.
Hence, there is no need to store the whole history of \(Y^{i}_{k}\).
Instead, the current state is sufficient for the update, which implies the storage space will not grow linearly with the step.
This property maintains if generators \(b\) and \(\sigma\) only depends on the expectation of a vector-value function of the state.

\section{Well-posedness of a class of path-distribution dependent SDEs}
In this section we discuss the solvability of equations \eqref{eqn-wmv} and \eqref{eqn-wemv} on any finite time interval \([0,T]\).
To cover more situations, we introduce the following path-distribution dependent SDE
	\begin{equation}	
		\label{eqn-gen}
		\md X_{t}  =b(X_{t},\cM_{t}[X])\md t+\sigma(X_{t},\cM_{t}[X])\md W_{t},
	\end{equation}
where $(\cM_{t}[X])_{t\ge 0}$ is a $\cP_2$-valued process depending on $X$.

To characterize $\cM_{t}[X]$ specifically, we introduce two spaces of processes as follows:
\[
\begin{aligned}
\cX_t & := \Big\{ X=(X_{s})_{0\leq s \leq t}: X_{s}\in\mathcal{F}_s\text{ for }0\leq s\leq t \text{ and } \sup_{0\leq s\leq t}\E|X_{s}|^{2}<\infty\Big\},\\
\cS_t & := \big\{ X\in \cX_t: X \text{ has a continuous modification}\big\}.
\end{aligned}
\]
\(\cX_t\) is a Banach space equipped with the norm
\[\| X \|_{\cX_t}:= \Big(\sup_{0\le s \le t}\E|X_{s}|^{2} \Big)^{1/2}.\]
But \(\cS_{t}\) is not complete under \(\|\cdot\|_{\cX_t}\);
alternatively, define the metric 
\begin{equation*}
	d_{\cS_t}=\|X-Y\|_{\cX_t} +\E\Big[\sup_{0\leq s \leq t}|X_{s}-Y_{s}|\wedge 1\Big].
\end{equation*}
Then \((\cS_t,d_{\cS_t})\) is a complete metric space.
There is a natural restriction from \(\cS_{T}\) to \(\cS_{t}\) for any \(t\leq T\).
We will not distinguish an element in \(\cS_{T}\) with its restriction.

\begin{asmp}\label{ass-M}
The mapping \(\cM:[0,T]\times\Omega\times \cS_{T}\to \cP_{2}\) satisfies:
\begin{itemize}
	\item For any \((t,X)\in[0,T]\times \cS_{T}\), \(\cM(t,\cdot,X)=\cM(t,\cdot,X|_{\cS_{t}})\).
\item \(\mathcal{M}\) is progressively measurable, i.e., for each \(t\) its restriction to \([0,t]\times\Omega\times \cS_{t}\) is \(\cB([0,t])\otimes\cF_{t}\otimes \cB(\cS_{t})\)-measurable.
 	\item There is \(\Omega' \subset \Omega\) with probability $1$ such that for any \((t,\omega)\in[0,T]\times \Omega'\),  \(\cM(t,\omega,\cdot)\) is continuous on \((\cS_T,d_{\cS})\), and \(\cM(\cdot,\omega,\cdot)\) is bounded on any bounded subset of \([0,T]\times \cS_{T}\).
	\item There is a constant \(K\geq 0\) such that 
	\begin{align*}
		\E[\cW_{2}(\cM(t,\cdot,X),\cM(t,\cdot,Y))^{2}]&\leq K\|X-Y\|^{2}_{\cX_t},\\
		\E[\cM(t,\cdot,X)(|\cdot|^{2})]&\leq K \|X\|^{2}_{\cX_t},
	\end{align*}
	for any \((t,X,Y)\in [0,T]\times \cS_{T}\times \cS_{T}\).
\end{itemize}
Denote \(\cM_{t}[X]=\cM(t,\cdot,X)\).
\end{asmp}

\begin{rmk}
	For any \(\pi\in\Pi\), \(\cL^{\pi}_{t}\) and \(\cE^{\pi}_{t}\) satisfy the above assumption.
	Moreover, all the mappings satisfying the above assumption construct a convex space. 
\end{rmk}

As for~\eqref{eqn-gen}, we relax Assumption \ref{assump} as follows.
\begin{asmp}
	\label{asmp-relax}
	$b$ and $\sigma$ are continuous functions on $\bR^d \times \cP_2$ and bounded on any bounded set; 
	there is a constant $ K \ge 0$ such that
	\begin{gather*}
		2\la b(x,\mu)-b(y,\nu),x-y \ra + \|\sigma(x,\mu) -\sigma(y,\nu)\|^{2}
		\leq K\big[|x-y|^{2}+\cW_{2}(\mu,\nu)^{2}\big],\\
	  2\la b(x,\mu),x\ra + \|\sigma(x,\mu)\|^{2} \leq K [ 1+|x|^{2} +\mu(|\cdot|^{2})],
	\end{gather*}
	for all \(x,\,y\in\bR^{d}\) and \(\mu,\,\nu\in \cP_2\).
	\end{asmp}

Then we have the following solvability result.
\begin{thm}\label{thm-well-posed}
	Under Assumptions~\ref{ass-M} and \ref{asmp-relax}, there exists a unique process \(X\) in \(\cS_{T}\) satisfying~\eqref{eqn-gen}
	 with the initial condition \(\xi \in L^{2}_{\cF_{0}}(\Omega)\).
\end{thm}

The proof needs the following result from Theorem 3.1.1 in \cite{prevot2007concise}.
\begin{lem}
	\label{lem-exist}
	Let Assumption~\ref{asmp-relax} be satisfied, 
	\(T>0\) be a fixed time, and \(\{ \mu_{t} \}_{t\ge 0}\) be a progressively measurable $\cP_{2}$-valued process such that \(\mu_{t}\) is uniformly bounded on \([0,T]\) a.s.
	Then, the following SDE 	\begin{equation*}
	    \begin{aligned}
			\md Z_{t} & =  b(Z_{t},\mu_{t})\md t + \sigma(Z_{t},\mu_{t})\md W_{t}, \\
			Z_{0} & =  \xi\in L^{2}_{\cF_{0}}(\Omega),
		\end{aligned}
	\end{equation*}
	has a unique continuous solution in \(\cS_{T}\).
\end{lem}

\begin{proof}[Proof of Theorem~\ref{thm-well-posed}]
	Set \(X^{0}_{t}=\xi\) for any \(t\in[0,T]\).
	We define \(X^{k+1}\) inductively by the following equation:
	\begin{equation}
		\label{eqn-t}
	X^{k+1}_{t} =\xi+\int_{0}^{t}b(X^{k+1}_{s},\cM_{s}[X^{k}])\md s+\int_{0}^{t}\sigma(X^{k+1}_{s},\cM_{s}[X^{k}])\md W_{s}.
	\end{equation}
	From Lemma \ref{lem-exist} and properties of \(\cM\), we have \(X^{k+1}\in \cS_{T}\) is well-defined.

	We will show that \((X^{k})_{k\geq 0}\) is a Cauchy sequence in \((\cS_T,d_{\cS_T})\).
	To simplify the notation, we denote 
	\[b(X^{m}_{t},\cM_{t}[X^{m-1}])-b(X^{n}_{t},\cM_{t}[X^{n-1}])\] 
	and 
	\[\sigma(X^{m}_{t},\cM_{t}[X^{m-1}])-\sigma(X^{n}_{t},\cM_{t}[X^{n-1}])\] 
	by \(\Delta b_{t}\) and \(\Delta \sigma_{t}\), respectively.
	By It\^o's formula, we have
	\begin{align*}
		  \md[ |X^{m}_{t}-X^{n}_{t}|^{2} ] 
		& = ( 2\la  X^{m}_{t}-X^{n}_{t},\Delta b_{t}\ra + \|\Delta \sigma\|^{2} )\md t + 2\la  X^{m}_{t}-X^{n}_{t}, \Delta\sigma_{t}\md W_{t} \ra.
	\end{align*}
	Together with Assumption~\ref{asmp-relax}, we have
    \begin{equation}\label{aln-ito}	
    \begin{aligned}
		|X^{m}_{t}-X^{n}_{t}|^{2} 
		& \leq \int_{0}^{t}  K \big[|X^{m}_{s}-X^{n}_{s}|^{2} + \cW_{2}(\cM_{s}[X^{m-1}],\cM_{s}[X^{n-1}])^{2}\big]\md s\\
		& \quad + 2 \int_{0}^{t} \la X^{m}_{s}-X^{n}_{s}, \Delta\sigma_{s}\md W_{s} \ra.
	\end{aligned}
	\end{equation}
	Setting \((m,n)=(k+1,k)\) and taking expectation on the both sides, we obtain
	\begin{align*}
		\E|X^{k+1}_{t}-X^{k}_{t}|^{2}
		& \leq \int_{0}^{t} K\big[  \E|X^{k+1}_{s}-X^{k}_{s}|^{2} + \E\big[\cW_{2}(\cM_{s}[X^{k}],\cM_{s}[X^{k-1}])^{2}\big] \big]\md s\\
		&\leq \int_{0}^{t} \big[ K \E|X^{k+1}_{s}-X^{k}_{s}|^{2}  + K^{2} \|X^{k}-X^{k-1}\|_{\cX_s} \big]\md s.
	\end{align*}
	Define \(g^{k+1}(t)=\me^{-Kt}\|X^{k+1}-X^{k}\|_{\cX_t}\).
	The above estimate implies
	\begin{align*}
		g^{k+1}(t)\leq  K^{2}\int_{0}^{t}  g^{k}(s) \md s.
	\end{align*}
	Following the standard Picard iteration, we have 
	\begin{equation}
		\label{eqn-pic}
		\|X^{k+1}-X^{k}\|_{\cX_t} \lesssim  \frac{(K^{2}t)^{k}\me^{Kt}}{k!}.
	\end{equation}
	Then, to show  \((X^{k})_{k\geq 0}\) is a Cauchy sequence in \( (\cS_T,d_{\cS_T})\), it suffices to show that for any \(\varepsilon>0\)
	\begin{equation*}
		\lim_{m,\,n\to\infty}\bP \Big(\sup_{0\leq t\leq T}|X_{t}^{m}-X_{t}^{n}|\geq \varepsilon \Big)=0.
	\end{equation*}
	For any \(m\) and \(n\), define \[\tau^{m,n}=T\wedge\inf\{t\geq 0:|X^{m}_{t}-X^{n}_{t}|\geq \varepsilon\}.\]
	Then, we have
	\begin{align*}
		\bP(\sup_{0\leq t\leq T}|X_{t}^{m}-X_{t}^{n}|\geq \varepsilon)
		&=\bP(|X_{\tau^{m,n}}^{m}-X_{\tau^{m,n}}^{n}|\geq \varepsilon)\\
		&\leq \varepsilon^{-2} \E |X^{m}_{\tau^{m,n}}-X^{n}_{\tau^{m,n}}|^{2}.
	\end{align*}
	By \eqref{aln-ito} and Doob's inequality, we have
	\begin{align*}
		\E |X^{m}_{\tau^{m,n}}-X^{n}_{\tau^{m,n}}|^{2} 
				& \leq \int_{0}^{T} \big[ K \E|X^{m}_{s}-X^{n}_{s}|^{2} +K \E\big[\cW_{2}(\cM_{s}[X^{m-1}],\cM_{s}[X^{n-1}])^{2}\big] \big]\md s\\
		&\leq KT \|X^{m}-X^{n}\|^{2}_{\cX_T} +K^{2}T \|X^{m-1}-X^{n-1}\|^{2}_{\cX_T}.
	\end{align*}
	This along with~\eqref{eqn-pic} yields that \((X^{k})_{k\geq 0}\) is a Cauchy sequence in $(\cS_T,d_{\cS_T})$.
	
	Denote by \(X\in\cS_T\) the limit of \((X^{k})_{k\geq 0}\).
	By taking a subsequence, we may assume that there is a full probability set $\Omega' \subset\Omega$ such that for any \(\omega\in\Omega'\),
	\[\lim_{k\to\infty}\sup_{0\leq t\leq T}|X^{k}_{t}-X_{t}|=0.\]
	By the continuity of \(b\) and \(\cM(t,\omega,\cdot)\),
	we have for any \((t,\omega)\in[0,T]\times \Omega'\),
	\[\lim_{k\to\infty}b(X^{k+1}_{t},\cM_{t}[X^{k}])=b(X_{t},\cM_{t}[X]).\]
	Since \(b\) and \(\cM(\cdot,\omega,\cdot)\) is bounded on any bounded set, by Lebesgue dominated convergence theorem, we have
	\[\lim_{k\to\infty}\int_{0}^{t}b(X^{k+1}_{s},\cM_{s}[X^{k}])\md s=\int_{0}^{t}b(X_{s},\cM_{s}[X])\md s \quad\text{ a.s.}\]

	Similarly, by the continuity of \(\sigma\) and \(\cM(t,\omega,\cdot)\),
	we have for any \((t,\omega)\in[0,T]\times \Omega'\),
	\[\lim_{k\to\infty}\sigma(X^{k+1}_{t},\cM_{t}[X^{k}])=\sigma(X_{t},\cM_{t}[X]).\]
	Define a sequence of stopping times 
	\[\tau_n:=T\wedge\inf \Big\{t\geq 0: \sup_{0\leq t\leq T}\|\sigma(X_{t},\cM_{t}[X])\|\geq n \Big\}.\]
	Again, by Lebesgue's dominated convergence theorem, we obtain
	\[\lim_{k\to\infty}\E\bigg(\int_{0}^{\tau_n}\|\sigma(X^{k+1}_{s},\cM_{s}[X^{k}])-\sigma(X_{s},\cM_{s}[X])\|^{2}\md s\bigg)=0.\]
	This implies, for any \(n>0\) and \(t\leq \tau_n\),
	\[\lim_{k\to\infty}\int_{0}^{t}\sigma(X^{k+1}_{s},\cM_{s}[X^{k}])\md W_{s}=\int_{0}^{t}\sigma(X_{s},\cM_{s}[X])\md W_{s}.\]
	Furthermore, since \(\sigma\) and \(\cM(\cdot,\omega,\cdot)\) is bounded on any bounded set, we have 
	\[\lim_{n \to\infty}\tau_n=T \quad\text{a.s.}\]
	Therefore, we show that for any \(t\in[0,T]\)
	\[
	X_{t}=\xi+\int_{0}^{t}b(X_{s},\cM_{s}[X])\md s+\int_{0}^{t}\sigma(X_{s},\cM_{s}[X])\md W_{s}.\]
	This yields the existence. 
	
	If $X^1, X^2 \in \cS_T$ both satisfy~\eqref{eqn-gen}
	with the same initial condition,
	then following a similar argument for~\eqref{aln-ito}, one has
	\[
		\begin{aligned}
		\E|X^{1}_{t}-X^{2}_{t}|^{2}
		& \leq \int_{0}^{t} K\big[  \E|X^{1}_{s}-X^{2}_{s}|^{2} + \E\big[\cW_{2}(\cM_{s}[X^{1}],\cM_{s}[X^{2}])^{2}\big] \big]\md s\\
		&\leq C \int_{0}^{t} \E|X^{1}_{s}-X^{2}_{s}|^{2}\big]\md s.
	\end{aligned}
	\]
	This along with Gronwall's inequality implies the uniqueness.
\end{proof}

\bibliography{mybib}

\end{document}